\documentclass[11pt,hyp,]{nyjm}
\usepackage{amsmath,amsthm,amssymb}
\usepackage{hyperref}
\hypersetup{nesting=true,debug=true,naturalnames=true}
\usepackage{graphicx,upref}
\usepackage{enumerate}

\let\<\langle
\let\>\rangle
\newcommand{\doi}[1]{doi:\,\href{http://dx.doi.org/#1}{#1}}

\let\uml\"

\title{Results on a Strong Multiplicity One Theorem} 

\author{Chandrasheel Bhagwat}
\address{A-414, Main Building, Indian Institute of Science Education and Research, Dr.\,Homi Bhabha Road, Pashan, Pune 411008, India} 
\email{cbhagwat@iiserpune.ac.in}

\author{Gunja Sachdeva}
\address{Department of Mathematics, BITS Pilani, K.K. Birla Goa Campus, Zuarinagar, Goa 403726, India}
\email{gunjas@goa.bits-pilani.ac.in}



\numberwithin{equation}{section}
\newtheorem{thm}[equation]{Theorem}
\newtheorem*{thm*}{Theorem}

\newtheorem*{cor*}{Corollary}
\newtheorem*{conj*}{Conjecture M}
\newtheorem{cor}[equation]{Corollary}
\newtheorem{lem}[equation]{Lemma}
\newtheorem{prop}[equation]{Proposition}

\theoremstyle{definition}

\newtheorem{rem}[equation]{Remark}
\newtheorem{defn}[equation]{Definition}

\def\R{\mathbb R}

\def\H{\mathcal H}

\def\Z{\mathbb Z}

\def\C{\mathbb C}

\def\la{\langle}
\newcommand{\ra}{\rangle}

\newcommand{\bi}{{\mathbf i}}

\def\O{\mathcal O}
\def\cC{\mathcal C}

\def\<{\langle}
\def\>{\rangle}

\def\SL{{\rm SL}}
\def\SO{{\rm SO}}

\begin{document} 
 
\begin{abstract}  
      We prove an analogue of the strong multiplicity one theorem in the context of $\tau_n$-spherical representations of the group $G = \SO(2,1)^\circ$ appearing in $L^2(\Gamma_i \backslash G)$ for uniform torsion-free lattices $\Gamma_i, i = 1, 2$ in $G$. This is a generalisation of a previous result by the first author and C. S. Rajan in \cite{B-R-2011} for the case of $G = \SO(2,1)^\circ$.
\end{abstract} 
\maketitle

\tableofcontents

%%%% **** The text of the paper starts here **** %%%%
\section{Introduction}
It is well known that the representation theory of semisimple Lie groups and the spectral theory of differential operators on associated symmetric and locally symmetric spaces are intimately related. In the context of discrete subgroup $\Gamma$ of a Lie group $G$ with compact quotient, the spectra are discrete, and hence it makes sense to study the multiplicity function for eigenvalues of Laplacian operator as well as the irreducible representations appearing in function spaces on $\Gamma \backslash G$.\smallskip

In \cite{B-R-2011}, C.S. Rajan and the first author established a strong multiplicity one property for $L^2(\Gamma \backslash G)$. In the same article, they have also studied the rank-1 semisimple Lie groups and  established a strong multiplicity one property for the $K$-spherical representations of rank-1 semisimple Lie group $G$, equivalently for Laplacian eigenvalues for $\cC^\infty(\Gamma \backslash G / K)$. \smallskip

The aim of this short note is to prove an analogue of the above results for the $\tau$-spherical representations of $\SO(2,1)^\circ$ for irreducible representations $\tau$ of its maximal compact subgroup $K \cong \SO(2)$. We use the classification of irreducible representations of $\SL(2,\R)$ and $\SO(2,1)^\circ$, their $K$-types, some properties of the function spaces on $\SO(2,1)^\circ$ with respect to the regular action of $K \times K$, and Selberg trace formula (as in \cite{B-R-2011}). \smallskip

Our main theorem is
\begin{thm*}\label{main result intro}
Let $G =  \SO(2,1)^\circ$, $\tau$ be an irreducible representation of $K$, and $\widehat{G}_{\tau}$ be the set of isomorphism classes of irreducible unitary representations $(\pi, V)$ of $G$ such that  the $\tau$-isotypic component $V^{\tau} \neq (0)$. Suppose $\Gamma_1$ and $\Gamma_2$ are uniform torsion-free lattices in $G$ such that the multiplicities of $\pi$ in $L^2(\Gamma_i \backslash G)$ for $i=1, 2$ satisfy
\[ m(\pi, \Gamma_1) = m(\pi, \Gamma_2) \quad \text
{for all but finitely many representations} \ \pi \in \widehat{G}_{\tau}.\] 
\[\text{Then,} \quad m(\pi, \Gamma_1) = m(\pi, \Gamma_2) \quad \text
{for all representations}\ \pi \in \widehat{G}_{\tau}.\]
\end{thm*}

We also establish the relation for the multiplicities of $\pi \in \widehat{G}_{\tau}$ in $L^2(\Gamma \backslash G)$ and multiplicities of eigenvalues of Laplacian in $\cC^\infty(\Gamma \backslash G, \tau)$ in Proposition ~\ref{multi-rep-eigen}. We note that the Laplacian eigenvalues corresponding to the non-trivial discrete series representations are negative. \smallskip

It is well-known that these eigenvalues are exactly the eigenvalues of the Laplacian operator acting on the smooth sections of the homogeneous vector bundles on $G/K$ defined by $\tau$, and hence all these eigenvalues must be positive. This bodes well with the conjecture that the  non-trivial discrete series representations do not appear in $L^2(\Gamma \backslash G)$.

\section{Representations of $\SO(2,1)^\circ$ and associated functions on hyperbolic plane}

\subsection{Lie groups}
%$\SO(2,1)$ and $PSL_2(\R)$}

Consider the real Lie group $\SO(2, 1)(\R) $ defined by
\[ \SO(2, 1)(\R) = \Bigg\{ g \in \SL_3(\R) \ : \ g^{T}
\left[\begin{array}{ccc}
     1& 0&0 \\
     0 &1&0 \\
     0&0&-1
\end{array}\right] g = \left[\begin{array}{ccc}
     1& 0&0 \\
    0 &1&0 \\
    0&0&-1
\end{array}\right]
\Bigg\}.\]

Let $ G = \SO(2, 1)^{\circ}$ be the connected component of $1_G$ in $\SO(2, 1)(\R) $. An Iwasawa decomposition of $G$ is given by $G = ANK$ where
\[
\begin{split}
A & = \Bigg\{a_t =  \left[\begin{array}{ccc}
    \cosh t & 0 &\sinh t \\
     0& 1&0\\
     \sinh t & 0 & \cosh t
\end{array}\right] \ : \ t \in \R\Bigg\},\\
N &= \Bigg\{n_u =  \left[\begin{array}{ccc}
    1 - \dfrac{u^2}{2} & u &\dfrac{u^2}{2} \\
     -u& 1&u\\
     -\dfrac{u^2}{2} & u & 1 + \dfrac{u^2}{2}
\end{array}\right] \ : \  u\in \R\Bigg\},\\
K & = \Bigg\{ k_\theta = \left[\begin{array}{ccc}
    \cos\theta & -\sin\theta &0\\
    \sin\theta & \cos\theta &0\\
    0&0&1
\end{array}\right] \ : \ \theta \in [0, 2\pi]\Bigg\}  = \Bigg\{ \left[\begin{array}{cc}
    B &  \\
     & 1
\end{array}\right] \ : \ B \in \SO(2)\Bigg\}.
\end{split}
\]

\noindent
{\bf An isomorphism between ${\rm PSL}_2(\R)$ and $\SO(2, 1)^\circ$}:
We fix an isomorphism between the groups ${\rm PSL}_2(\R)$ and $\SO(2, 1)^\circ$ once and for all. \\

Define $\Psi : \SL_2(\R) \longrightarrow \SO(2, 1)^\circ$ by
\[
\Psi\left( \left[\begin{array}{cc}
    a & b \\
    c & d
\end{array}\right]\right) = \left[ \begin{array}{ccc}
\dfrac{1}{2}(a^2 - b^2 -c^2+d^2)     &  ab-cd & \dfrac{1}{2}(a^2 + b^2 -c^2-d^2) \\
    ac-bd & ad+bc & ac+bd\\
    \dfrac{1}{2}(a^2 - b^2 +c^2-d^2)& ab+cd & \dfrac{1}{2}(a^2 + b^2 +c^2+d^2)  
\end{array}\right].
\]
It can be checked that $\Psi$ is a group homomorphism, and kernel($\Psi$) =  $(\pm I_2)$. Thus 
$\Psi$ defines an isomorphism of Lie groups between ${\rm PSL}_2(\R)$ and $\SO(2, 1)^\circ$. We also note down the images of some special elements of $\SL_2(\R) $ under $\Psi$ here.

\[
\begin{split}
\Psi \left(  \left[ \begin{array}{cc} e^t & 0 \\ 0 & e^{-t} \end{array} \right] \right)
= a_{2t} = \left[ \begin{array}{ccc}
    \cosh 2t & 0 &\sinh 2t \\
     0& 1&0\\
     \sinh 2t & 0 & \cosh 2t
\end{array}\right] \quad \forall \ t \in \R, \\
\Psi  \left( \left[ \begin{array}{cc} 1 & u \\ 0 & 1 \end{array}\right] \right)
= n_u = \left[\begin{array}{ccc}
    1 - \dfrac{u^2}{2} & u &\dfrac{u^2}{2} \\
     -u& 1&u\\
     -\dfrac{u^2}{2} & u & 1 + \dfrac{u^2}{2}
\end{array}\right]  \quad \forall \ u \in \R, \\
\Psi \left(  \left[ \begin{array}{cc} \cos \theta & -\sin \theta \\ \sin \theta & \cos \theta \end{array}\right] \right)
=  k_{2 \theta} = \left[\begin{array}{ccc}
    \cos 2\theta & -\sin 2\theta &0\\
    \sin 2\theta & \cos 2\theta &0\\
    0&0&1
\end{array}\right]  \quad \forall \ \theta \in \R.
\end{split}
\]
Thus the Iwasawa decompositions for ${\rm SL}_2(\R)$ and $\SO(2, 1)^\circ$ are related via the map $\Psi$. We make use of this fact later in the discussion for $K$-types of irreducible unitary representations of $\SO(2, 1)^\circ$.
\smallskip

\subsection{Induced representations of $G$}

Fix Haar measures on $G$ and $K$ such that $\int \limits_K 1~ dk  = 1$. We identify $dk$ with the normalised Lebesgue measure $\dfrac{d\theta}{2 \pi}$ on $[0, 2 \pi]$. Let $s \in \C$. Define a character of the torus $A$ of $G =\SO(2,1)^\circ$ by 
\[ \mu_s: a_t \mapsto e^{st} \quad \forall \ t \in \R. \]

Inflate $\mu_s$ to the Borel subgroup $B: = NA = AN$ of $G$ and then consider the induced space $V(s)$ consisting of all measurable functions 
$f: G \rightarrow \C$ such that $f_{|_{K}} \in L^2(K, dk)$ and
\[ f(b g) = e^t \mu_s(a_t) f(g) \quad \forall \ b = a_t n_u \in B = AN, \ g \in G.\]
Define a representation $\rho_s$ of $G$ on $V(s)$ by
\[ \rho_s(g)f(y) = f(yg) \quad \forall \ g, y \in G, \ f \in V(s). \]
The representation $(\rho_s, V(s))$ is unitary and irreducible precisely when $s \in \bi \R \cup (-1, 1)$ (see \cite{Lang}). Furthermore, the only possible isomorphisms between $(\rho_s, V(s))$ and $(\rho_s', V(s'))$   are when $s = \pm s'$. Thus we can without loss of generality assume $s \in \bi \R_{\geq 0} \sqcup (0,1)$ while considering the irreducible unitary representation $V(s)$, and we call $(\rho_s, V(s))$ by the names {\it principal series} when $s \in \bi \R_{\geq 0} $ and {\it complementary series} when $s \in (0,1)$, respectively.
\smallskip

\subsection{Associated spherical functions on $\H$}
Fix $s \in \C$. Let $\H = \{ z = x + \bi y: y > 0 \}$ be the hyperbolic $2$-space. Consider the usual action of $\SL(2,\R)$ on $\H$ by fractional linear transformations and let $g: z \mapsto g\cdot z: =  \Psi^{-1} (g) \cdot z$ be the corresponding action of $G =\SO(2,1)^\circ$ on $\H$. We use this to define a function $\chi_s: \H \rightarrow \C$ by
\[ \chi_s (x + \bi y) = \chi_s( n_x a_{\ln y} \cdot \bi) = \mu_s(n_x a_{\ln y}) =  y^s \quad \forall \ x+ \bi y \in \H. \]
Let $\phi_s$ be the function defined by
\[ \phi_s(z) = \int \limits_{K} \chi_s(k \cdot z) ~dk \quad \forall \ z \in \H,
\]
where $k \cdot z: = \Psi^{-1} (k) \cdot z \ \text{for all} \ z \in \H, \ k \in K$, and $\Psi^{-1} (k) \cdot z$ denotes the usual action of $\SL(2,\R)$ on $\H$ by fractional linear transformations.\smallskip

A simple calculation (as in \cite[p.65-67]{Lub}) gives that

\begin{lem} For every $s \in \C$, the function $\phi_{s+\frac{1}{2} }$ satisfies:\smallskip

\begin{enumerate}[(i)]
\item $ \phi_{s+\frac{1}{2}}(k \cdot z) = \phi_{s+\frac{1}{2}}(z) \quad \forall \ k \in K, \ z \in \H $.

\smallskip 

\item $\phi_{s+\frac{1}{2}} \in \mathcal C^{\infty}(\H)$ and $ \Delta \phi_{s+\frac{1}{2}} = \frac{1-s^2}{4}  \phi_{s+\frac{1}{2}}$ where $\Delta = -y^2 (\frac{d^2}{dx^2} +  \frac{d^2}{dy^2} )$ is the hyperbolic Laplacian operator on $\H$.

\end{enumerate}
\end{lem}
\smallskip

\subsection{Discrete series representations}

For every positive integer $m \geq 2$, there are inequivalent unitary discrete series representations $D^{\pm}(m)$ of the group $G$ that can be realised as an irreducible sub-representation of the induced space $V(m-1)$. This follows from the construction of discrete series representations of $\SL(2,\R)$ (see \cite{Lang}).
\smallskip

\subsection{$K$-types of irreducible representations of $\SO(2,1)^\circ$}

Let $n \in \Z$. Define $\tau_n$ to be a character of the (abelian) group $K$ by 
\[ \tau_n(k_\theta) = e^{\bi n \theta} \quad \forall \ k_\theta \in K. \]

For any unitary irreducible representation $(\pi, V)$ of $G$, we have a Hilbert direct sum decomposition of $V$ into $K$-isotypes
\[ V = \widehat{ \bigoplus \limits_{n \in \Z}} H_n,\]
where $H_n$ is the $K$-isotype $V^{\tau_n}$ of $V$ w.r.t. $\tau_n$, i.e.
\[ H_n = \{ v \in V: \pi(k_\theta) v = \tau_n(k_\theta) v = e^{\bi n \theta} v \quad \forall \ k_\theta \in K \} .
\]

It can be shown that for every $n \in \Z$, the complex vector space $H_n$ is either $(0)$ or of dimension $1$ (this is a well-known fact. As reference, see \cite[p. 24]{Lang}).
 \vspace{0.3cm}

 We say that the $K$-type $\tau_n$ appears in $\pi$ if $H_n \neq (0)$ in $V$. Thus
\[ V = \widehat{\bigoplus \limits_{n \in \Z}} H_n =  \widehat {\bigoplus \limits_{\tau_n \ \text{appears in} ~\pi}} H_n. \]

The irreducible unitary representations $\pi$ of $G$ bijectively correspond to the irreducible unitary representations $\tilde{\pi}$ of ${\rm SL}(2, \R)$ such that $-I_{2 \times 2} \in {\rm kernel}(\tilde{\pi})$. Thus only `half' of the discrete series  representations of ${\rm SL}(2, \R)$ appear in the list for $G$.  Using this and the adjoint action of the Lie algebra $\mathfrak g$ of $G$, we can compute the $K$-types that appear in $\pi$ for all irreducible unitary representations $\pi$ of $G$. The following table summarises this information.

\[
\begin{array}{|c|c|}
\hline
1     & \tau_0 \\
\hline
D_m^+, \ \ m \geq 2 \ \text{ even}     & \tau_n:  \ \ n = \frac{m}{2}+j, \ j\geq 0 \\
\hline 
D_m^-, \ \ m \leq -2 \ \text{ even}     & \tau_n:  \ \ n = -\frac{m}{2}-j, \ j\geq 0\\
\hline
\rho_s , \ \ s \in \bi \mathbb{R}_{\geq 0} \sqcup (0,1) & \tau_n: \ \ n \in \mathbb{Z}\\
\hline
\end{array}
\]

\begin{defn}An irreducible unitary representation $(\pi, V)$ of $G$ is said to be $\tau_n$-spherical if there exists a nonzero vector $v \in V$ such that 
\[
\pi(k_\theta)v = \tau_n(k_\theta)v = e^{\bi n\theta}v \quad \forall \ k_\theta \in K.
\]

\end{defn}

For a given $n \in \Z$, we can list the $\tau_n$-spherical representations of $G$ as follows:

\[
\begin{array}{|c|c|}
\hline
n > 0 & 1,  D_m^+: m \in \{ 2, 4, \ldots, 2n \}, \rho_s: s  \in  \bi \mathbb{R}_{\geq 0} \sqcup (0,1)  \\
\hline
n = 0 & 1, \rho_s: s  \in  \bi \mathbb{R}_{\geq 0} \sqcup (0,1)  \\
\hline
n < 0 & 1,  D_m^-: m \in \{ -2, -4, \ldots, -2n \}, \rho_s: s  \in  \bi \mathbb{R}_{\geq 0} \sqcup (0,1)  \\
\hline
\end{array}\]

\smallskip

\noindent{\bf Notation:} Let $\widehat{G}_{\tau_n}$ denote the $\tau_n$-spherical spectrum of $G$, defined as the set of all equivalence classes of irreducible unitary representations $\pi$ of $G$ such that $\tau_n$ appears in $\pi$.

\subsection{$\tau_n$-spherical representations and $\tau_n$-spherical  functions}

For $n \in \Z$, we denote by $\cC_c^\infty(G \slash \slash K, \tau_n)$, the space of all compactly supported smooth functions $f: G \rightarrow \C$ such that 
\[ 
f(k_{\theta_1} x k_{\theta_2} ) = e^{\bi n (\theta_1+ \theta_2)} f(x) \quad \forall \ x \in G, k_{\theta_1}, k_{\theta_2} \in K.
\]
Let $(\pi, V)$ be a $\tau_n$-spherical representation of $G$. Let $e_\pi \in H_n = V^{\tau_n} \subseteq V$ be a unit vector. Let $\phi_\pi:  G \rightarrow \C$ be a function defined by
\[ \phi_\pi(x) = \la \pi(x) e_\pi, e_\pi \ra \quad \forall \ x \in G.\]
%Thus it follows that $\phi_\pi \in \cC_c^\infty(G \slash \slash K, \tau_n)$.\smallskip
We state a lemma here which will be used later.

\begin{lem}\label{character} Let $(\pi, V)$ be a $\tau_n$-spherical representation of $G$. Then the character distribution of $\pi$ satisfies 
\[ \chi_\pi(f) = \int \limits_G f(g) \phi'_\pi(g) dg \quad \forall \ f \in \cC_c^\infty(G \slash \slash K, \tau_n),
\]
where $\phi'_\pi(g)$ is defined by 
\[
\phi'_\pi(g) =
\begin{cases}
 \langle \pi(x) v_{-n}, v_{-n} \rangle &  \text{if} \ H_{-n} = \C v_{-n} \neq (0), ||v_{-n}|| =1\\
 0 & \text{if} \ H_{-n} = (0).
\end{cases}
\]
\end{lem}

\begin{proof}
Let $k_{\theta} \in K$, $v, w \in V$, $f \in \cC_c^\infty(G \slash \slash K, \tau_n)$. We have
\[
\begin{split}
\langle \pi(k_{\theta}) \pi(f) v, w  \rangle & = \langle  \pi(f) v, \pi(k_{-\theta})w  \rangle\\
& = \int \limits_G f(x)~  \langle  \pi(x) v, \pi(k_{-\theta})w  \rangle ~dx \\
& = \int \limits_G f(k_{-\theta}x) ~  \langle  \pi(x) v,w  \rangle  ~dx\\
& = e^{- \bi n \theta} \langle \pi(f)v, w \rangle .
\end{split}
\]
Thus we conclude that 
\[\pi(f)v \in V^{\tau_{-n}} = H_{-n} \quad  \forall \ v \in V, \ f \in \cC_c^\infty(G \slash \slash K, \tau_n).\]

\begin{itemize}
    \item Case (1). The isotypic component $V^{\tau_{-n}} = H_{-n}$ in $(\pi, V)$ is non-zero, say equals $\C v_{-n}$ for some $v_{-n}$ with $|| v_{-n} || = 1$. Thus $ \chi_\pi(f)$ =  trace $\pi(f) = \la \pi(f) v_{-n}, v_{-n} \ra$ since the components $H_m$ are mutually orthogonal and $\pi(f)v \in H_{-n} \ \forall \ v \in V$.\smallskip

\item Case (2). The  isotypic component $V^{\tau_{-n}} = H_{-n}$ in $(\pi, V)$ is zero. In this case $\pi(f) = 0$ from above calculation.\smallskip
\end{itemize}

In both cases, we get the desired conclusion for the character distribution $\chi_\pi$ of $\pi$.
\end{proof}
From Lemma ~\ref{character} and its proof, we also conclude that

\begin{cor}
If $\pi$ is $\tau_n$-spherical, and $f \in  \cC_c^\infty(G \slash \slash K, \tau_{-n})$, then 
\[ \chi_\pi(f) = \int \limits_G f(x) \phi_\pi(x) ~dx.
\]
\end{cor}

\section{Some preliminary results}
Let $\Gamma$ be a uniform lattice in $G$ so that the quotient space $\Gamma \backslash G$ is compact. Consider the regular representation of $G$ on $L^2(\Gamma \backslash G)$. For an irreducible unitary representation $\pi$ of $G$, let $m(\pi, \Gamma)$ be the multiplicity of $\pi$ in the decomposition of $L^2(\Gamma \backslash G)$. 
\smallskip

Now further assume that  $\Gamma$ is torsion-free. The center $Z( U_{\mathfrak g})$ of the universal enveloping algebra of $G$ acts on the space of all smooth functions on $\Gamma \backslash G$. Let $\cC^\infty(\Gamma \backslash G, \tau_n)$ be the subspace of $\cC^\infty(\Gamma \backslash G)$ consisting of all smooth right $\tau_n$-equivariant $f$ i.e.
\[ f(g k_\theta) = e^{\bi n \theta} f(g) \quad \forall \ g \in G, \ k_\theta \in K. \]
It can be checked that $\cC^\infty(\Gamma \backslash G, \tau_n)$ is stable under $Z(U_{\mathfrak g})$-action, and the eigenspaces $V(\lambda, \Gamma, \tau_n)$, w.r.t. characters $\lambda$ of $Z(U_{\mathfrak g})$ are finite dimensional subspaces of 
$\cC^\infty(\Gamma \backslash G, \tau_n)$ (see \cite{Lub}). We discuss the following result that relates the multiplicities  $m(\pi, \Gamma_i)$ for $\pi  \in \widehat{G}_{ \tau_n}$ with these eigenspaces.\smallskip

If $(\pi, V_\pi)$ is a $\tau_n$-spherical representation, then we can associate a character $\lambda_\pi$ of $Z(U_{\mathfrak g})$ such that $\phi_\pi(x): = \la \pi(x) e_\pi, e_\pi \ra $ defines an eigenfunction of $\lambda_\pi$ where $e_\pi$ is a vector in $V_\pi$ with $|| e_\pi || = 1$.

\begin{lem} If $\pi_j: 1 \leq j \leq r$ are mutually inequivalent $\tau_n$-spherical representations then $\phi_{\pi_j}: 1 \leq j \leq r$ are linearly independent and analytic functions on $G$.
\end{lem}

\begin{proof}
The conclusion of this lemma follows at once from the fact that $\phi_{\pi_j}$s are eigenvectors of distinct characters and also of elliptic essentially self-adjoint differential operators in $Z(U_{\mathfrak g})$.
\end{proof}

\begin{prop}\label{multi-rep-eigen} If $\pi$ is a $\tau_n$-spherical representation, then 
\[
m(\pi, \Gamma) = {\rm \dim}_\C V(\lambda_\pi, \Gamma, \tau_n).\]
Conversely, if $\lambda$ is a character of $Z(U_{\mathfrak g})$ such that $V(\lambda, \Gamma) \neq 0$, then there exists a unique $\tau_n$-spherical representation $\pi$ of $G$ such that $\lambda = \lambda_\pi$ and $m(\pi, \Gamma) = {\rm \dim}_\C V(\lambda_\pi, \Gamma, \tau_n)$.
\end{prop}

\begin{proof}
Let $m = m(\pi, \Gamma)$. Let $W_j: 1 \leq j \leq m$ be mutually orthogonal subspaces in $L^2(\Gamma \backslash G)$, each of them isomorphic to $\pi$. Let $e_\pi^{(j)} \in W_j$ be a unit vector such that it is a basis of the one dimensional $\tau_n$-isotype $W_j^{\tau_n}$. Define $h_j$ by
\[ h_j(x) = \int \limits_{[0, 2\pi]} e^{- \bi n \theta} e_\pi^{(j)}(xk_\theta) \frac{d\theta}{2 \pi} \quad \forall \ x \in G.\]
These functions $h_j$s are smooth, are in $\cC^\infty(\Gamma \backslash G, \tau_n)$,  and are eigenfunctions of the character $\lambda_\pi$ of $Z(U_{\mathfrak g})$. Since the vectors $e_\pi^{(j)} \in W_j$ and are mutually orthogonal, it  follows that functions $h_j$s are also mutually orthogonal and hence we have
\[ m(\pi, \Gamma) \leq {\rm \dim}_\C V(\lambda_\pi, \Gamma, \tau_n).\]
Conversely, if $f \in \cC^\infty(\Gamma \backslash G, \tau_n) $ is an eigenfunction of a character $\lambda$ of $Z(U_{\mathfrak g})$, then $f \in L^2(\Gamma \backslash G)$. Write 
\[ L^2(\Gamma \backslash G) = \widehat{\bigoplus \limits_{j}} V_j, \]
where each $V_j$ is an irreducible subspace of  $L^2(\Gamma \backslash G)$ corresponding to $\pi_j \in \widehat{G}$.
Hence $f$ can be written as
\[ f = \sum \limits_{j} \alpha_j v_j,\]
where each $v_j$ is a unit vector in $V_j$.
Observe that, we in fact get a sum over only those $j$ such that $V_j^{\tau_n} \neq 0$. Further, since $f$ is an eigenfunction of $\lambda$, it follows that $v_j$ is also an eigenfunction of $\lambda$ for all $j$ such that $\alpha_j \neq 0$. \smallskip
Using the linear independence of characters $\lambda_{\pi_j}$ for $\pi_j \in \widehat{G}$, the sum is over only those $V_j$ for which the representation is $\pi$ and the associated character is $\lambda_\pi = \lambda$. This  implies that 
$$m(\pi, \Gamma)= {\rm \dim}_\C V(\lambda_\pi, \Gamma, \tau_n).$$
\end{proof}

\section{Main results}
Let $\Gamma_1$ and $\Gamma_2$ be uniform torsion-free lattices in $G$.  We prove the following analogue of the strong multiplicity theorem for $\tau_n$-spherical spectra for $\Gamma_1$ and $ \Gamma_2$.

\begin{thm}\label{main result}
Let $n \in \Z$. Suppose $\Gamma_1$ and $\Gamma_2$ are uniform torsion-free lattices in $G$ such that
\[
m(\pi, \Gamma_1) = m(\pi, \Gamma_2) \quad \text
{for all but finitely many representations} \ \pi \in \widehat{G}_{\tau_n}.
\]
\[ \text{Then} \quad m(\pi, \Gamma_1) = m(\pi, \Gamma_2) \quad \text
{for all representations}\ \pi \in \widehat{G}_{\tau_n}.\]
\end{thm}
We first prove a series of lemmas here.

\begin{lem} \label{sep-points} 
The space $\cC_c^\infty(G // K,  \tau_n)$ separates points on the orbits of $K \times K$-action on $G$ given by 
\[ (k_{\theta_1}, k_{\theta_2})  \cdot g : = k_{\theta_1}  g k_{\theta_2}^{-1}. \] 
\end{lem}

\begin{proof} Let $KxK$ and $KyK$ be disjoint orbits of the above action. Let $f$ be a smooth and compactly supported function such that $f = 0$ on $KyK$ and $f>0$ on $KxK$. Define $F$ by 
\[
F(g) = \int \limits_{[0, 2\pi]}~  \int \limits_{[0, 2\pi]}~
e^{-\bi n (\theta_1 + \theta_2)} f(k_{\theta_1}  g k_{\theta_2})   \frac{d\theta_1}{2 \pi}  \frac{d\theta_2}{2 \pi}.
\]
This $F$ is in $\cC_c^\infty(G // K,  \tau_n)$, and separates 
$KxK$ and $KyK$ because of the assumed properties of the function $f$.
\end{proof}
We state two more results about uniform lattices in semisimple Lie groups here. We refer the reader to \cite{B-R-2011} for the proofs.

\begin{lem} \label{disjoint}
If $\Gamma$ is a torsion-free uniform lattice in $G$, then every nontrivial conjugacy class $[\gamma]_G$ of $\gamma \in \Gamma$ is disjoint from $K$ and hence $1_G \notin K [\gamma]_G K $.
\end{lem}

%\begin{proof} \textcolor{blue}{Chandrasheel's comments: This result can be cited as it is from Bhagwat-Rajan paper. Proof is not needed again in my opinion.}

\begin{lem} \label{stable open B}
Given torsion-free uniform lattices $\Gamma_1$ and  $\Gamma_2$ in $G$, there exists a non-empty open set $B \subset G$ which is $(K \times K)$-stable and disjoint from all conjugacy classes $[\gamma]_G$ with $\gamma \in \Gamma_1 \cup \Gamma_2$.
\end{lem}

%\begin{proof} \textcolor{blue}{Chandrasheel's comments: This result can be cited as it is from Bhagwat-Rajan paper. Proof is not needed again in my opinion.}
%
%\end{proof}

\begin{proof} [{\bf Proof of Theorem ~\ref{main result}}]

Let $S$ be a finite subset of $\widehat{G}_{\tau_n}$ such that
\[ m(\pi, \Gamma_1) = m(\pi, \Gamma_2) \quad \forall\ \pi \in \widehat{G}_{\tau_n} \setminus S. \]
Applying the Selberg trace formula for lattices $\Gamma_1$ and $\Gamma_2$ for a function $f \in \cC_c^\infty(G // K, \tau_{-n})$, we see that 

\[ \sum \limits_{\pi \in S} \left(m(\pi, \Gamma_1) - m(\pi, \Gamma_2) \right) \chi_\pi(f) = 
\sum \limits_{[\gamma] \in [\Gamma_1]_G \cup [\Gamma_2]_G}
[a(\gamma, \Gamma_1) - a(\gamma, \Gamma_2)] O_\gamma(f)
\]
(here the term $O_\gamma(f)$ refers to the orbital integral of $f$ w.r.t. the left invariant measure on the conjugacy class $[\gamma]$ in $G$ (see \cite{B-R-2011}, \cite{Wal} for more details). Let $\phi =\sum \limits_{\pi \in S} \left(m(\pi, \Gamma_1) - m(\pi, \Gamma_2) \right) \phi_\pi$. Thus using Lemma ~\ref{character}, we have \smallskip
\[ \int \limits_{G} f(g) \ \phi(g) \ d\mu(g) = 
\sum \limits_{[\gamma] \in [\Gamma_1]_G \cup [\Gamma_2]_G}
[a(\gamma, \Gamma_1) - a(\gamma, \Gamma_2)] O_\gamma(f).
\]
Choose a $K \times K$-stable open set $B$ that avoids all conjugacy classes appearing on the right hand side of above (as in Lemma ~\ref{stable open B}). For every $f \in \cC_c^\infty(G // K, \tau_{-n})$ that is supported on $B$, we have
\[ \int \limits_{G} f(g)\ \phi(g) \ d\mu(g) = 0, \]

showing that $\phi = 0$ on $B$ since $\cC_c^\infty(G // K, \tau_{-n})$ separates points on $B$ as seen in Lemma ~\ref{sep-points}. Since all $\phi_\pi$ with $\pi \in S$ are analytic and linearly independent, we conclude that
\[ m(\pi, \Gamma_1) = m(\pi, \Gamma_2) \quad \forall\ \pi \in S.\]
\end{proof}

\begin{rem} We expect that the analogous result to Theorem ~\ref{main result} should also hold for the case of real rank one Lie groups $G$ and the $K$-types of their irreducible unitary representations with respect to their maximal subgroups $K$. The authors would like to explore more in that direction as a sequel to this work. It will be also interesting to ask whether the isospectrality with respect to $\tau$-spherical representations with respect to a particular $K$-type determines the isospectrality with respect to the full representation spectrum of $L^2(\Gamma_i \backslash G)$, $i =1,2$.
\end{rem}

\begin{rem}
As pointed out by the anonymous referee of this manuscript, these results
can be generalised to the case of the group $\SL_2(\R)$ without any non-trivial modification
of the proof.  There are two cases to consider.  If the uniform lattice contains the central torsion element $-I_{2\times 2}$ but no other torsion, then this case reduces to the
case already studied of ${\rm PSL}_2(\R)$ divided by a torsion-free uniform lattice.
If however the uniform lattice in $\SL_2(\R)$  is torsion-free, then one will see in addition more
representations with half-integral $K$-types, namely non-spherical principal series, discrete series, and limits of discrete series, as described in \cite{Lang}.
\end{rem}

\section*{Acknowledgements}
\noindent
We thank the anonymous referee for their valuable comments. Gunja Sachdeva is supported by DST-SERB POWER Grant No. SPG/2022/001738.


\begin{thebibliography}{CS79}
\bibitem[BR11]{B-R-2011}
{\scshape Bhagwat, C.; Rajan C.\ S.},
On a spectral analog of the strong multiplicity one theorem.
{\em Int. Math. Res. Not.} (2011), no. 18, 4059--4073.
\mrev{2836013},
\zbl{1226.22014},
\doi{10.1093/imrn/rnq243}.

\bibitem[LA85]{Lang}
{\scshape Lang, Serge.}
${\rm SL}_2(\R)$.
Graduate Texts in Mathematics, \textbf{105} {\em Springer-Verlag, New York.} (1985).
\mrev{0803508},
\doi{10.1007/978-1-4612-5142-2}.

\bibitem[LU10]{Lub}
{\scshape Lubotzky, Alexander.}
Discrete groups, expanding graphs and invariant measures. With an appendix by Jonathan D. Rogawski.
{\em Modern Birkh\"{a}user Classics. Birkh\"{a}user Verlag.} (2010), 
\mrev{2569682},
\doi{10.1007/978-3-0346-0332-4}.



\bibitem[WA76]{Wal}
{\scshape Wallach, Nolan R.}
On the Selberg trace formula in the case of compact quotient.
{\em Bull. Amer. Math. Soc.} {\bf 82} (1976), no. 2, 171--195.
\mrev{0404533},
\zbl{0351.22008},
\doi{10.1090/S0002-9904-1976-13979-1}.



\end{thebibliography}
\end{document}